\documentclass[leqno]{amsart}

\usepackage{amsmath,amssymb,amsthm}
\usepackage{mathrsfs}
\usepackage{braket}
\usepackage{graphicx}
\usepackage{amscd}
\usepackage[all]{xy}
\usepackage{comment}
\usepackage{eucal}

\theoremstyle{definition}
\newtheorem{dfn}{Definition}[section]
\newtheorem{lem}[dfn]{Lemma}
\newtheorem{cor}[dfn]{Corollary}
\newtheorem{prp}[dfn]{Proposition}
\newtheorem{theom}[dfn]{Theorem}
\newtheorem{rem}[dfn]{Remark}
\newtheorem{fct}[dfn]{Fact}
\newtheorem{ex}[dfn]{Example}

\newtheorem{prob}[dfn]{Problem}
\newcommand{\mf}{\mathfrak}
\DeclareMathOperator{\End}{\sf End}

\DeclareMathOperator{\Hom}{\sf Hom}

\DeclareMathOperator{\str}{\sf str}
\DeclareMathOperator{\ch}{\sf ch}
\DeclareMathOperator{\id}{\sf id}
\DeclareMathOperator{\Ind}{\sf Ind}
\DeclareMathOperator{\ad}{\sf ad}
\DeclareMathOperator{\Ad}{\sf Ad}
\DeclareMathOperator{\Ker}{\sf Ker}
\DeclareMathOperator{\Res}{\sf Res}
\DeclareMathOperator{\spn}{\sf span}

\makeatletter

\@addtoreset{equation}{section}
\makeatother

\title[On Resolutions of highest weight modules]{On Resolutions of highest weight modules over the $\mathcal{N}=2$ superconformal algebra}
\author{Shinji KOSHIDA}
\address[Shinji KOSHIDA]{Department of Basic Science, The University of Tokyo, Tokyo, Japan 153-8902}
\email{koshida@vortex.c.u-tokyo.ac.jp}

\author{Ryo SATO}
\address[Ryo SATO]{Institute of Mathematics, Academia Sinica, Taipei, Taiwan 10617}
\email{sato@gate.sinica.edu.tw}


\begin{document}

\maketitle

\begin{abstract}
In this paper we construct Bernstein--Gelfand--Gelfand type resolution of
simple highest weight modules
over the simple $\mathcal{N}=2$ vertex operator superalgebra
of central charge $c_{p,p'}=3\left(1-\frac{2p'}{p}\right)$
by means of the Kazama--Suzuki coset construction.
As an application, 
we compute the twisted Zhu algebras of 
the simple $\mathcal{N}=2$ vertex operator superalgebra.
We also compute the Frenkel--Zhu bimodule structure associated with
a certain simple highest weight module of central charge $c_{3,2}=-1$.
\end{abstract}

\section{Introduction}

\subsection{Background}

One of the most fundamental problems in the study of $2$-dimensional conformal field theory
is to compute fusion rules between primary fields
(see e.g.\,\cite{verlinde1988fusion}).
In \cite{feigin1994fusion}, B. Feigin and F. Malikov computed
the fusion rules of the Wess--Zumino--Witten model
associated with the affine Lie algebra $\widehat{\mf{sl}}_{2}$ at Kac--Wakimoto admissible levels.
Though this model has the remarkable modular invariance property (see \cite{KW88} for detail), 
a naive generalization of the Verlinde formula,
which is originally proposed in \cite{verlinde1988fusion}, fails in general
(see \cite{koh1988fusion}, \cite{mathieu1990fractional}, \cite{awata1992fusion}).
After more than two decades of studies (see \cite{CR12} for historical detail), in \cite{creutzig2013modular}, T.\,Creutzig and D.\,Ridout conjectured
a consistent generalization of the Verlinde formula (partially proved in \cite[Corollary 7.7]{creutzig2017braided})
which includes some non-standard simple $\widehat{\mf{sl}}_{2}$-modules,
called relaxed highest weight modules. 
This type of $\widehat{\mf{sl}}_{2}$-module was
initially studied by B.\,Feigin, A.\,Semikhatov, and I.\,Tipunin \cite{FST98}
in connection with the $\mathcal{N}=2$ superconformal algebra
 by the Kazama--Suzuki coset construction.
In fact, as was observed in \cite{FST98}, relaxed highest weight $\widehat{\mf{sl}}_{2}$-modules 
turn out to correspond to standard $\mathcal{N}=2$ highest weight modules.
Recently, in \cite{sato2017modular}, the second author proposed
a conjectural Verlinde formula in the $\mathcal{N}=2$ side,
which is a counterpart to Creutzig--Ridout's formula in the $\widehat{\mf{sl}}_{2}$ side.
It is worth noting that, by some technical reasons, the computation of fusion rules in the 
$\mathcal{N}=2$ side seems to be easier than that in the $\widehat{\mf{sl}}_{2}$ side.
However, the structure of the fusion rules has not yet been determined 
in the both sides, except for few cases.
For example, see \cite{ridout2011fusion} for a detailed analysis in the case of $\widehat{\mf{sl}}_{2}$
at level $k=-\frac{1}{2}$.

In terms of vertex operator superalgebras, 
fusion rules can be estimated by the theory of Frenkel--Zhu's bimodule developed in \cite{frenkel1992vertex} (see also \cite{kac1994vertex}, \cite{li1999determining}).
In the theory, twisted Zhu's algebra, which is originally introduced by Y.\,Zhu in \cite{zhu1996modular} and generalized by several authors (e.g.\,\cite{kac1994vertex}, \cite{dong1998twisted}, \cite{xu1998introduction}, \cite{dong1998vertex}, \cite{de2006finite}), plays a prominent role.
Unlike in the case of the Virasoro algebra, unfortunately,
highest weight modules over the $\mathcal{N}=2$ superconformal algebra 
(even in the Neveu--Schwarz sector) may contain subsingular vectors\footnote{Let $I$ be the proper submodule of a highest weight module $\mathcal{M}$
generated by singular vectors.
Then a non-singular vector in $\mathcal{M}$ is called subsingular if
its image in $\mathcal{M}/I$ is singular.
}
discovered in 
\cite{gato1997chiral}.
See e.g.\,\cite[\S1]{dorrzapf1998embedding} for further information.
Due to such a circumstance,
the twisted Zhu algebra of the simple $\mathcal{N}=2$ vertex operator superalgebra 
(see Appendix \ref{N2VOSA} for the definition) has not yet been completely determined 
to the best of our knowledge. 
See \cite[\S3]{eholzer1997unitarity} for partial results.

\subsection{Main Results}\label{mainresults}
In this paper we study the repressentation theory of 
the $\mathcal{N}=2$ vertex operator superalgebra
and determine its twisted Zhu algebra(s)
as a first step towards the full understanding of the corresponding Verlinde formula.

Our main tool is an exact functor 
induced by the Kazama--Suzuki coset construction
(see \cite[\S4.1]{sato2016equivalences} for the definition).
More precisely, this functor is defined as one from a certain full subcategory of $\widehat{\mf{sl}}_{2}$-modules of level $k\in\mathbb{C}\setminus\{-2\}$ to a similar full subcategory of modules over the $\mathcal{N}=2$ superconformal algebra of central charge $c=\frac{3k}{k+2}$.
From this view point, 
the structure of $\mathcal{N}=2$ Verma modules turns out to reflect that of relaxed Verma modules over $\widehat{\mf{sl}}_{2}$ with spectral flow twists
(see \cite{FST98}, \cite{sato2016equivalences} for more detail).

In the present paper, the level $k$ is supposed to be a Kac--Wakimoto admissible level
and the corresponding central charge $c$ is given by
$$c=c_{p,p'}:=3\left(1-\frac{2p'}{p}\right),$$ 
where $(p,p')$ is a pair of coprime positive integers such that $p\geq2$.
Then our main results in \S3 are summarized as follows (see Remark \ref{Main1ex}).
\begin{theom}\label{Main1}
Let $L_{c}$ be the simple $\mathcal{N}=2$ vertex operator superalgebra
of central charge $c=c_{p,p'}$.
Then every simple $L_{c}$-module admits Bernstein--Gelfand--Gelfand (BGG) type resolution in terms of 
spectral flow twisted modules (see Appendix \ref{SF} for the definition) of certain Verma type modules.
\end{theom}
As a corollary, in \S\ref{uSV}, we prove the following (see Theorem \ref{Uniqueness} for detail).
\begin{theom}\label{Main2}
Every highest weight module over the $\mathcal{N}=2$ superconformal algebra 
of central charge $c_{p,p'}$ corresponding to 
a generalized Verma module for $\widehat{\mf{sl}}_{2}$
(see Lemma \ref{weyl} for the precise definition) contains a unique singular vector
which generates its maximum proper submodule.
\end{theom}

In \S\ref{app}, we give two applications of the above singular vectors.
First, for general $c=c_{p,p'}$, we determine the twisted Zhu algebras $A_{\sigma}(L_{c})$ and $A_{\id}(L_{c})$
(see \S\ref{DefZhuSec} for the definition), where
$\sigma$ and $\id$ are the parity involution and the identity automorphism of $L_{c}$, respectively
(see Theorem \ref{Kernel1} and Corollary \ref{Kernel2} for detail).

\begin{theom}\label{Main3} For $c=c_{p,p'}$, we have the following.

\begin{enumerate}
\item
The $\sigma$-twisted Zhu algebra $A_{\sigma}(L_{c})$ is isomorphic to
$U(\mf{gl}_{1|1})/\langle \phi_{c}\rangle_{\sf ideal}$
for some $\phi_{c}\in U(\mf{gl}_{1|1})$ corresponding to $\mathcal{N}_{p,p'}$.
\item
The $\id$-twisted Zhu algebra $A_{\id}(L_{c})$ is isomorphic to
$\mathbb{C}[{\sf h},{\sf q}]/\langle f_{c},g_{c}\rangle_{\sf ideal}$
for some $f_{c},g_{c}\in \mathbb{C}[{\sf h},{\sf q}]$
corresponding to $\mathcal{N}_{p,p'}$ and $G_{-\frac{1}{2}}^{+}G_{-\frac{1}{2}}^{-}\mathcal{N}_{p,p'}$.
\end{enumerate}
\end{theom}
Note that the latter has been conjectured by W.\,Eholzer and M.\,Gaberdiel in \cite[\S3]{eholzer1997unitarity}. See Remark \ref{EG1} and Remark \ref{EG2} for more detail.

Second, in the case of $c=c_{3,2}=-1$, 
we compute the Frenkel--Zhu bimodule structure associated with a certain simple highest weight module.
We should mention that the simple vertex operator superalgebra $L_{c}$ is not $C_{2}$-cofinite if $p'\neq1$ by \cite[Corollary 2.3]{sato2017modular} (see also Corollary \ref{finiteness} in this paper).
As a consequence of the computation, we obtain a non-trivial upper bound for conjectural fusion rules in \cite[Example 5.2]{sato2017modular}. See Remark \ref{conjFR} for detail.


\medskip
{\bf Acknowledgments: }
The second author would like to thank
Naoki Genra and Yoshiyuki Koga
for valuable comments.
The first author was supported by a Grant-in-Aid for JSPS Fellows (Grant No.\,17J09658).

\section{BGG type resolution for simple $\widehat{\mf{sl}}_{2}$-modules}

In this section we give a brief review on the existence of several BGG type resolutions for 
certain simple weight $\widehat{\mf{sl}}_{2}$-modules.
At the end of this section, we give a new resolution for certain relaxed highest weight
$\widehat{\mf{sl}}_{2}$-modules in terms of relaxed Verma modules (see \S\ref{RVM} for the definition).

\subsection{Notations}
We fix a pair of coprime integers 
$(p,p')\in\mathbb{Z}_{\geq2}\times\mathbb{Z}_{\geq1}$
and set $a:=\frac{\,p'}{p}$.
Unless otherwise specified, the letter $k$ stands for the Kac--Wakimoto admissible level
$-2+a^{-1}$ of the affine Lie algebra $\widehat{\mf{sl}}_{2}=\mf{sl}_{2}\otimes\mathbb{C}[t^{\pm1}]\oplus\mathbb{C}K$.
We denote the standard $\mf{sl}_{2}$-triple by $\{E,F,H\}$.
We put the following two finite sets:
\begin{align*}
&\mathcal{I}_{\sf KW}:=\big\{(m,n)\in\mathbb{Z}^{2}\,\big|\,1\leq m\leq p-1,\,0\leq n\leq p'-1\big\},\\
&\mathcal{I}_{\sf BPZ}:=\big\{(m,n)\in\mathcal{I}_{\sf KW}\,\big|\,n\neq0,\,p'm+pn\leq pp'\big\}.
\end{align*}

For $j\in\mathbb{C}$,
the Verma module $M_{j,k}$ of $\widehat{\mf{sl}}_{2}$
is the $\widehat{\mf{sl}}_{2}$-module
which is freely generated by
a vector $\ket{j,k}$ subject to the relations
$$X_{n}\ket{j,k}=E_{0}\ket{j,k}=0,\ H_{0}\ket{j,k}=2j\ket{j,k},\ K\ket{j,k}=k\ket{j,k}$$
for any $X\in\mf{sl}_{2}$ and $n\in\mathbb{Z}_{>0}$, where $X_{n}:=X\otimes t^{n}$.
It is clear that we have
$$L_{0}\ket{j,k}=\Delta(j)\ket{j,k}\ \ \left(\Delta(j):=\frac{j(j+1)}{k+2}\right),$$ where the operator $L_{0}$ is given by the Sugawara construction.

Finally we set
$$M(m,n):=M_{j_{m,n},k}\ \ 
\left(j_{m,n}:=\frac{m-1}{2}-\frac{n}{2a}\right)$$
for $m,n\in\mathbb{Z}$.
Denote by $L(m,n)$ the simple quotient
of $M(m,n)$.

\subsection{Parabolic BGG Resolution}

We first recall the following BGG resolution 
of the Kac--Wakimoto admissible highest weight $\widehat{\mf{sl}}_{2}$-module $L(r,s)$ 
for $(r,s)\in\mathcal{I}_{\sf KW}$ constructed by F. Malikov.

\begin{theom}\cite[Theorem A]{malikov1991verma}\label{MalikovBGG}
We set $M_{n}:=M(n)\oplus M(-n)$ for $n\in\mathbb{Z}_{>0}$, where
$M(2m):=M(2pm+r,s)$ and $M(2m-1):=M(2pm-r,s)$
for $m\in\mathbb{Z}$.
Then there exists an exact sequence
\begin{equation}\label{BGG}
\cdots
\rightarrow M_{n}
\rightarrow \cdots
\rightarrow M_{2}
\rightarrow M_{1}
\rightarrow M(r,s)
\rightarrow L(r,s)
\rightarrow 0.
\end{equation}
\end{theom}

Next, for later use, we consider the parabolic version of the above BGG resolution.
For $m\in\mathbb{Z}_{>0}$,
we define the generalized Verma module $V(m)$
to be the $\widehat{\mf{sl}}_{2}$-module freely generated
by the vector $\ket{m}$ subject to the relations
$$X_{n}\ket{m}=E_{0}\ket{m}=F_{0}^{m}\ket{m}=0,\ H_{0}\ket{m}=(m-1)\ket{m},\ K\ket{m}=k\ket{m}$$
for any $X\in\mf{sl}_{2}$ and $n\in\mathbb{Z}_{>0}$.
Then the following parabolic BGG resolution of $L(r,0)$ for $(r,0)\in\mathcal{I}_{\sf KW}$ is obtained as 
a slight generalization of Theorem \ref{MalikovBGG}.

\begin{theom}\label{MalikovgBGG}
We set 
$V_{2m}:=V(2pm+r)$ and $V_{2m-1}:=V(2pm-r)$
for $m\in\mathbb{Z}_{>0}$.
Then there exists an exact sequence
\begin{equation}\label{gBGG}
\cdots
\rightarrow V_{n}
\rightarrow \cdots
\rightarrow V_{2}
\rightarrow V_{1}
\rightarrow V(r)
\rightarrow L(r,0)
\rightarrow 0.
\end{equation}
\end{theom}

\begin{rem}
When $p'=1$, 
the resolutions (\ref{BGG}) and (\ref{gBGG})
for $L(r,0)$ are obtained by H. Garland and J. Lepowsky
in \cite[Theorem 8.6]{garland1976lie} (see also \cite{rocha1982projective}).
\end{rem}

\subsection{Relaxed Verma modules}\label{RVM}

In this subsection we consider $k\in\mathbb{C}$.
Relaxed Verma modules are firstly introduced by \cite{FST98}
in the case of $\widehat{\mf{sl}}_{2}$ (cf.\,\cite{futorny1996irreducible}).
Their quotient modules, which are called relaxed highest weight modules, have been studied in
\cite{ridout2015relaxed}, \cite{arakawa2017weight}, \cite{adamovic2017realizations}, \cite{kawasetsu2018relaxed}.
For the reader's convinience, we specify the precise definition 
of relaxed Verma modules used in this paper (cf. \cite[\S2.1]{ridout2015relaxed}, \cite[\S2]{kawasetsu2018relaxed}).

Let $\mf{g}$ be a simple complex Lie algebra (or a basic classical Lie superalgebra, in general),
$\mf{h}$ be its Cartan subalgebra, and $U_{0}$ be the centralizer of $\mf{h}$ in $U(\mf{g})$.
For a simple finite-dimensional $U_{0}$-module $M$,
we define the action of 
$\widehat{\mf{g}}_{\geq0}:=\mf{g}\otimes \mathbb{C}[t]
\oplus\mathbb{C}K$ on 
the induced $U(\mf{g})$-module
$U(\mf{g})\otimes_{U_{0}} M$
by $X_{n}\mapsto \delta_{n,0}X$ and $K\mapsto k\id$ for $X\in\mf{g}$ and
$n\in\mathbb{Z}_{\geq0}$.
We denote this $\widehat{\mf{g}}_{\geq0}$-module by $\widetilde{M}_{k}$.
Then we call the parabolic induction
$\widehat{M}_{k}:=\Ind_{\widehat{\mf{g}}_{\geq0}}^{\widehat{\mf{g}}}\widetilde{M}_{k}$
the relaxed Verma module induced from $M$.
When the level $k$ is not critical, it is clear that the Virasoro algebra acts on $\widehat{M}_{k}$
via the Sugawara construction
and the operator $L_{0}$ acts diagonally on $\widehat{M}_{k}$.

When $\mf{g}=\mf{sl}_{2}$ and $k\in\mathbb{C}\setminus\{-2\}$, 
the algebra $U_{0}$ is freely generated by $H$ and the quadratic Casimir element and
it is not hard to verify that the definition of relaxed Verma modules can be rephrased as follows
(see \cite[\S3.3.2]{sato2016equivalences} for detail):
\begin{dfn}
For $(h,j)\in\mathbb{C}^{2}$, 
the relaxed Verma module $R_{h,j,k}$
is the $\widehat{\mf{sl}}_{2}$-module freely generated by
a relaxed highest weight vector $\ket{h,j,k}$ subject to the relations
\begin{align*}
&X_{n}\ket{h,j,k}=0,\ L_{0}\ket{h,j,k}=h\ket{h,j,k},\\
&H_{0}\ket{h,j,k}=2j\ket{h,j,k},\ K\ket{h,j,k}=k\ket{h,j,k}
\end{align*} for any $X\in\mf{sl}_{2}$ and
$n\in\mathbb{Z}_{>0}$.
We write $L_{h,j,k}$ for its unique simple quotient module.
\end{dfn}

\subsection{Twisting functor}

In this subsection we recall some known methods 
(see e.g.\,\cite[\S 2.4]{kashiwara1998kazhdanposi}) to construct relaxed Verma modules from ordinary Verma modules.

Let $U_{F}$ be the (right) localization of the non-commutative algebra $U:=U(\widehat{\mf{sl}}_{2})$ with respect to the multiplicative set 
$\{(F_{0})^{n}\,|\,n\in\mathbb{Z}_{\geq0}\}$ in $U$.
For $j\in\mathbb{C}$, we define the functor ${\sf F}^{j}$ by
$${\sf F}^{j}:=\Res^{U_{F}}_{U}\circ\Ad(F^{j}_{0})^{*}\circ\Ind_{U}^{U_{F}}\colon 
U\text{\sf-mod}\to U\text{\sf-mod},$$
where $\Ad(F^{j}_{0})$ is the algebra automorphism of $U_{F}$
defined by
$$\Ad(F^{j}_{0})(u):=\sum_{\ell\geq0}\binom{j}{\ell}\ad(F_{0})^{\ell}(u)F_{0}^{-\ell}\in U_{F}$$
for $u\in U_{F}$.
Since the functor ${\sf F}^{j}$ is the composition of three exact functors, it is exact.
For an $\widehat{\mf{sl}}_{2}$-module $M$, we denote each element $v$ 
of the underlying space $U_{F}\otimes_{U}M$ of the twisted module
${\sf F}^{j}M:={\sf F}^{j}(M)$ by ${\sf F}^{j}v$.

\begin{lem}\label{twisted}
Let $j_{0},j\in\mathbb{C}$. 
Then there exists a unique $\widehat{\mf{sl}}_{2}$-module homomorphism
\begin{equation}\label{twistedisom}
R_{\Delta(j_{0}),j,k}\to {\sf F}^{-j+j_{0}}M_{j_{0},k};\ 
\ket{\Delta(j_{0}),j,k}\mapsto {\sf F}^{-j+j_{0}}\ket{j_{0},k}.
\end{equation}
Moreover it is an isomorphism if and only if $j\notin\{j_{0},-j_{0}-1\}+\mathbb{Z}_{\leq0}$.
\end{lem}

\begin{proof}
Since the vector ${\sf F}^{-j+j_{0}}\ket{j_{0},k}$
satisfies the same annihilation relation as that of $\ket{\Delta(j_{0}),j,k}$,
the $\widehat{\mf{sl}}_{2}$-module homomorphism (\ref{twistedisom}) uniquely exists.
By a direct computation, we have
$$E_{0}{\sf F}^{-j+j_{0}}F_{0}^{-m}\ket{j_{0},k}
=-(m+j-j_{0})(m+j+j_{0}+1){\sf F}^{-j+j_{0}}F_{0}^{-m-1}\ket{j_{0},k}$$
for $m\in\mathbb{Z}_{>0}$.
Since the image of (\ref{twistedisom}) is 
given by $U(\mf{sl}_{2}\otimes\mathbb{C}[t^{-1}]){\sf F}^{-j+j_{0}}\ket{j_{0},k}$
and the set
$\big\{{\sf F}^{-j+j_{0}}F_{0}^{n}\ket{j_{0},k}\big|\,n\in\mathbb{Z}\big\}$
forms a free $U(\mf{sl}_{2}\otimes\mathbb{C}[t^{-1}]t^{-1})$-basis of ${\sf F}^{-j+j_{0}}M_{j_{0},k}$,
the mapping (\ref{twistedisom}) is bijective if and only if
$j\notin\{j_{0},-j_{0}-1\}+\mathbb{Z}_{\leq0}$.
\end{proof}

The following lemma is used in the next subsection.

\begin{lem}\label{homog}
Assume that $j\notin\{j_{0},-j_{0}-1\}+\mathbb{Z}$.
Then, for any $n\in\mathbb{Z}$, the module ${\sf F}^{-(j+n)+j_{0}}M_{j_{0},k}$
is isomorphic to
$R_{\Delta(j_{0}),j,k}$.
\end{lem}

\begin{proof}
By the previous lemma, it suffices to prove that $R_{\Delta(j_{0}),j+n,k}$ is isomorphic to $R_{\Delta(j_{0}),j,k}$. This follows from direct calculation.
\end{proof}

\subsection{Relaxed BGG Resolution}
Since $-j_{m,n}-1=j_{p-m,p'-n}$, we have by Lemma \ref{twisted}
$$R_{m,n;j}:=R_{\Delta(j_{m,n}),j,k}\simeq {\sf F}^{-j+j_{m,n}}M(m,n)$$
if $j\notin\{j_{m,n},j_{p-m,p'-n}\}+\mathbb{Z}_{\leq0}$.
Then, by applying the exact functor ${\sf F}^{-j+j_{r,s}}$ to the BGG resolution (\ref{BGG}),
we obtain the following BGG type resolution for
the simple quotient $L_{r,s;j}$ of $R_{r,s;j}$:

\begin{prp}
Assume that $(r,s)\in\mathcal{I}_{\sf BPZ}$ and $j\notin\{j_{r,s},j_{p-r,p'-s}\}+\mathbb{Z}$.
We set $R_{n}:=R(n;j)\oplus R(-n;j)$ for $n\in\mathbb{Z}_{>0}$,
where
$R(2m;j):=R_{2pm+r,s;j}$ and $R(2m-1;j):=R_{2pm-r,s;j}$
for $m\in\mathbb{Z}$.
Then there exists an exact sequence
\begin{equation}\label{rBGG}
\cdots
\rightarrow R_{n}
\rightarrow \cdots
\rightarrow R_{2}
\rightarrow R_{1}
\rightarrow R_{r,s;j}
\rightarrow L_{r,s;j}
\rightarrow 0.
\end{equation}
\end{prp}

\begin{proof}
Since ${\sf F}^{-j+j_{r,s}}M_{n}$ is isomorphic to $R_{n}$ by Lemma \ref{homog},
there exists an exact sequence
$$\cdots
\rightarrow R_{n}
\rightarrow \cdots
\rightarrow R_{2}
\rightarrow R_{1}
\rightarrow R_{r,s;j}
\rightarrow {\sf F}^{-j+j_{r,s}}L(r,s)
\rightarrow 0.$$
Therefore the formal character of ${\sf F}^{-j+j_{r,s}}L(r,s)$
is given by
\begin{equation}\label{CRchar}
\ch {\sf F}^{-j+j_{r,s}}L(r,s)=\sum_{n\in\mathbb{Z}}(-1)^{n}\ch R_{n},
\end{equation}
where $R_{0}:=R_{r,s;j}$.
It is clear that the formal series (\ref{CRchar}) coincides with the character of $L_{r,s;j}$ given by T. Creutzig and D. Ridout in 
\cite[Corollary 5]{creutzig2013modular}. 
Since Creutzig--Ridout's character formula
is proved in \cite[Proposition 5.4]{adamovic2017realizations} for some special cases and in \cite[Theorem 5.2]{kawasetsu2018relaxed} for general cases\footnote{The character formula in \cite[Corollary 5]{creutzig2013modular} also follows from \cite[Theorem 3.1]{sato2017modular} together with \cite[Proposition 7.1 (4) and Theorem 7.8]{sato2016equivalences}.}, this completes the proof.
\end{proof}

\section{BGG type resolution for simple $\mf{ns}_{2}$-modules}\label{relaxed}

In this section we prove the existence of BGG type resolution for 
certain simple highest weight modules over the $\mathcal{N}=2$ superconformal algebra.
As we mentioned in \S\ref{mainresults}, we use the functor 
$\Omega_{j}^{+}\colon\mathcal{C}_{\sf af}\to\mathcal{C}_{\sf sc}$
for $j\in\mathbb{C}$ defined in \cite[\S4.1]{sato2016equivalences},
where $\mathcal{C}_{\sf af}$ and $\mathcal{C}_{\sf sc}$
are certain full subcategories of $\widehat{\mf{sl}}_{2}$-modules of level $k$
and $\mf{ns}_{2}$-modules of central charge $c$, respectively.
Our discussion below is essentially based on the exactness of the functor $\Omega_{j}^{+}$
(see \cite[Theorem 4.4]{sato2016equivalences}) and the spectral flow equivariance property
(see \cite[Corollary 6.3]{sato2016equivalences}).

\subsection{Notations}

The Neveu-Schwarz sector of the $\mathcal{N}=2$ 
superconformal algebra is the Lie superalgebra
\begin{equation*}
\mf{ns}_{2}=
\bigoplus_{n\in\mathbb{Z}}\mathbb{C}
L_{n}\oplus
\bigoplus_{n\in\mathbb{Z}}\mathbb{C}
J_{n}\oplus
\bigoplus_{r\in\mathbb{Z}+\frac{1}{2}}\mathbb{C}
G^{+}_{r}\oplus
\bigoplus_{r\in\mathbb{Z}+\frac{1}{2}}\mathbb{C}
G^{-}_{r}\oplus
\mathbb{C}C
\end{equation*}
whose $\mathbb{Z}_{2}$-grading is given by 
\begin{equation*}
(\mf{ns}_{2})^{\bar{0}}=
\bigoplus_{n\in\mathbb{Z}}\mathbb{C}
L_{n}\oplus
\bigoplus_{n\in\mathbb{Z}}\mathbb{C}
J_{n}\oplus\mathbb{C}C,\ \ 
(\mf{ns}_{2})^{\bar{1}}=
\bigoplus_{r\in\mathbb{Z}+\frac{1}{2}}\mathbb{C}
G^{+}_{r}\oplus
\bigoplus_{r\in\mathbb{Z}+\frac{1}{2}}\mathbb{C}
G^{-}_{r}
\end{equation*}
with the following 
(anti-)commutation relations:
\begin{flalign*}
&
[L_{n},L_{m}]=(n-m)L_{n+m}+\frac{1}{12}(n^{3}-n)C\delta_{n+m,0},\\ 
&
[L_{n},J_{m}]=-mJ_{n+m}, \ 
[L_{n},G^{\pm}_{r}]=\left(\frac{n}{2}-r\right)G^{\pm}_{n+r}, \\
&
[J_{n},J_{m}]=\frac{n}{3}C\delta_{n+m,0},\ \ 
[J_{n},G^{\pm}_{r}]=\pm G^{\pm}_{n+r}, \\
&
[G^{+}_{r},G^{-}_{s}]=
2L_{r+s}+(r-s)J_{r+s}+
\frac{1}{3}\left(r^{2}-\frac{1}{4}\right)C\delta_{r+s,0}, \\
&
[G^{+}_{r},G^{+}_{s}]=
[\,G^{-}_{r},G^{-}_{s}]=0,\ \ 
[\mf{ns}_{2},C]=\{0\},
\end{flalign*}
for $n,m\in\mathbb{Z}$ and $r,s\in\mathbb{Z}+\frac{1}{2}$.
Define a triangular decomposition $\mf{ns}_{2}=
(\mf{ns}_{2})_{+}\oplus
(\mf{ns}_{2})_{0}\oplus
(\mf{ns}_{2})_{-}$
by
\begin{equation*}
\begin{array}{lcl}
(\mf{ns}_{2})_{+} & := &
\displaystyle
\bigoplus_{n>0}\mathbb{C}L_{n}\oplus
\bigoplus_{n>0}\mathbb{C}J_{n}\oplus
\bigoplus_{r>0}\mathbb{C}G^{+}_{r}\oplus
\bigoplus_{r>0}\mathbb{C}G^{-}_{r},\\
(\mf{ns}_{2})_{-} & := &
\displaystyle
\bigoplus_{n<0}\mathbb{C}L_{n}\oplus
\bigoplus_{n<0}\mathbb{C}J_{n}\oplus
\bigoplus_{r<0}\mathbb{C}G^{+}_{r}\oplus
\bigoplus_{r<0}\mathbb{C}G^{-}_{r}
,\\
(\mf{ns}_{2})_{0} & := &
\displaystyle
\mathbb{C}L_{0}\oplus
\mathbb{C}J_{0}
\oplus\mathbb{C}C.
\end{array}
\end{equation*}

For $(h,j,c)\in\mathbb{C}^{3}$, 
the Verma module $\mathcal{M}_{h,j,c}$ of $\mf{ns}_{2}$
is the $\mathbb{Z}_{2}$-graded $\mf{ns}_{2}$-module freely generated by
an even vector $\ket{h,j,c}^{\sf sc}$ subject to the relations
\begin{align*}
&(\mf{ns}_{2})_{+}\ket{h,j,c}^{\sf sc}:=\{0\},\ L_{0}\ket{h,j,c}^{\sf sc}:=h\ket{h,j,c}^{\sf sc},\\ 
&J_{0}\ket{h,j,c}^{\sf sc}:=j\ket{h,j,c}^{\sf sc},\ C\ket{h,j,c}^{\sf sc}:=c\ket{h,j,c}^{\sf sc}.
\end{align*}
Similarly, the chiral Verma module $\mathcal{M}_{j,c}^{+}$
is the $\mathbb{Z}_{2}$-graded $\mf{ns}_{2}$-module freely generated by
an even vector $\ket{j,c}^{\sf sc}$ subject to the relations
\begin{align*}
&(\mf{ns}_{2})_{+}\ket{j,c}^{\sf sc}:=\{0\},\ G_{-\frac{1}{2}}^{+}\ket{j,c}^{\sf sc}:=0,\\
&L_{0}\ket{j,c}^{\sf sc}:=\frac{j}{2}\ket{j,c}^{\sf sc},\ 
J_{0}\ket{j,c}^{\sf sc}:=j\ket{j,c}^{\sf sc},\ C\ket{j,c}^{\sf sc}:=c\ket{j,c}^{\sf sc}.
\end{align*}

\subsection{BGG type resolution for atypical modules}

For $m,n\in\mathbb{Z}$, we set
$$\mathcal{M}^{+}(m,n):=\mathcal{M}^{+}_{2aj_{m,n},3(1-2a)}$$
and denote its simple quotient by $\mathcal{L}(m,n)$.
For the sake of completeness, 
we recall the following BGG type resolution of $\mathcal{L}(r,s)$ for $(r,s)\in\mathcal{I}_{\sf KW}$, which is firstly given 
by B. Feigin, A. Semikhatov, V. Sirota, and I. Tipunin in \cite[Theorem 3.1]{FSST99}
(see \cite[Theorem 7.11]{sato2016equivalences} for the proof).

\begin{theom}[{\cite{FSST99}, \cite{sato2016equivalences}}]\label{FsstBGG}
We set $\mathcal{M}^{+}_{n}:=\mathcal{M}^{+}(n)\oplus \mathcal{M}^{+}(-n)$ for $n>0$,
where
$\mathcal{M}^{+}(2m):=\mathcal{M}^{+}(2pm+r,s)^{pm}$
and $\mathcal{M}^{+}(2m-1):=\mathcal{M}^{+}(2pm-r,s)^{pm-r}$
for $m\in\mathbb{Z}$.
Then there exists an exact sequence
\begin{equation}\label{tBGG}
\cdots
\rightarrow \mathcal{M}_{n}^{+}
\rightarrow \cdots
\rightarrow \mathcal{M}_{2}^{+}
\rightarrow \mathcal{M}_{1}^{+}
\rightarrow \mathcal{M}^{+}(r,s)
\rightarrow \mathcal{L}(r,s)
\rightarrow 0
\end{equation}
of $\mf{ns}_{2}$-modules\footnote{In the rest of this section, for simplicity of notation, we do not consider
the $\mathbb{Z}_{2}$-graded structure given in the previous subsection.}.
\end{theom}

In what follows, we construct a new BGG type resolution of
$\mathcal{L}(r,0)$, which is the counterpart
of (\ref{gBGG}) in the $\mathcal{N}=2$ side.
For convinience of later use, we give an explicit description of 
$\Omega_{j_{m,0}}^{+}\big(V(m)\big)$ as follows.

\begin{lem}\label{weyl}
For $m\in\mathbb{Z}_{>0}$, the $\mf{ns}_{2}$-module 
$\Omega_{j_{m,0}}^{+}\big(V(m)\big)$
is isomorphic to the $\mf{ns}_{2}$-module $\mathcal{V}(m)$ freely generated by
$\ket{m}^{\sf sc}$ subject to the relations
\begin{align*}
&(\mf{ns}_{2})_{+}\ket{m}^{\sf sc}:=\{0\},\ G^{+}_{-\frac{1}{2}}\ket{m}^{\sf sc}
:=G^{-}_{-\frac{2m-1}{2}}\cdots G^{-}_{-\frac{3}{2}}G^{-}_{-\frac{1}{2}}\ket{m}^{\sf sc}:=0,\\ 
&L_{0}\ket{m}^{\sf sc}:=aj_{m,0}\ket{m}^{\sf sc},\ J_{0}\ket{m}^{\sf sc}:=2aj_{m,0}\ket{m}^{\sf sc},\ C\ket{m}^{\sf sc}:=3(1-2a)\ket{m}^{\sf sc}.
\end{align*}
\end{lem}

Since the proof is same as that of \cite[Proposition 7.1]{sato2016equivalences}), 
we omit it.

\begin{rem}
The $\mf{ns}_{2}$-module $\mathcal{V}(1)$ is by definition isomorphic to the vacuum $\mf{ns}_{2}$-module $V_{c}$ of central charge $c=c_{p,p'}$ (see Appendix \ref{N2VOSA} for the definition).
\end{rem}

\begin{theom}\label{chBGG2}
We set 
$\mathcal{V}_{2m}:=
\mathcal{V}(2pm+r)^{pm}$ and
$\mathcal{V}_{2m-1}:=
\mathcal{V}(2pm-r)^{pm-r}$
for $m\in\mathbb{Z}_{>0}$.
Then there exists an exact sequence
\begin{equation}\label{gBGG2}
\cdots
\rightarrow \mathcal{V}_{n}
\rightarrow \cdots
\rightarrow \mathcal{V}_{2}
\rightarrow \mathcal{V}_{1}
\rightarrow \mathcal{V}(r)
\rightarrow \mathcal{L}(r,0)
\rightarrow 0.
\end{equation}
\end{theom}

\begin{proof}
By \cite[Example 6.4]{sato2016equivalences}, we have 
$$\Omega^{+}_{j_{r,0}}(V_{2m})
=\Omega^{+}_{j_{2pm+r,0}-pm}\big(V(2pm+r)\big)
\simeq\Omega^{+}_{j_{2pm+r,0}}\big(V(2pm+r)\big)^{pm}$$ for any $m\in\mathbb{Z}_{>0}$.
Then, by Lemma \ref{weyl}, we obtain
$\mathcal{V}_{2m}\simeq\Omega^{+}_{j_{r,0}}(V_{2m})$.
We also obtain
$\mathcal{V}_{2m-1}\simeq\Omega^{+}_{j_{r,0}}(V_{2m-1})$
in the same way.
This completes the proof.
\end{proof}

\subsection{BGG type resolution for typical modules}

For $(m,n)\in\mathbb{Z}^{2}$ and $j\in\mathbb{C}$, we set
$$\mathcal{M}_{m,n;j}:=\mathcal{M}_{\Delta(j_{m,n})-aj^{2},2aj,3(1-2a)}$$
and denote its simple quotient by $\mathcal{L}_{m,n;j}$.
Note that by \cite[Proposition 7.1 (1) and (4)]{sato2016equivalences} we have
$\mathcal{M}_{m,n;j}\simeq\Omega^{+}_{j}(R_{m,n;j})$
and $\mathcal{L}_{m,n;j}\simeq\Omega^{+}_{j}(L_{m,n;j})$.

By applying the exact functor $\Omega^{+}_{j}$ to the BGG type resolution (\ref{rBGG}),
we obtain the following resolution of $\mathcal{L}_{r,s;j}$:

\begin{theom}\label{rBGG2}
Assume that $(r,s)\in\mathcal{I}_{\sf BPZ}$ and $j\notin\{j_{r,s},j_{p-r,p'-s}\}+\mathbb{Z}$.
We set $\mathcal{M}_{n}:=\mathcal{M}(n;j)\oplus \mathcal{M}(-n;j)$ for $n\in\mathbb{Z}_{>0}$,
where
$\mathcal{M}(2m;j):=(\mathcal{M}_{2pm+r,s;j})^{pm}$
and $\mathcal{M}(2m-1;j):=(\mathcal{M}_{2pm-r,s;j})^{pm-r}$
for $m\in\mathbb{Z}$.
Then there exists an exact sequence
$$\cdots
\rightarrow \mathcal{M}_{n}
\rightarrow \cdots
\rightarrow \mathcal{M}_{2}
\rightarrow \mathcal{M}_{1}
\rightarrow \mathcal{M}_{r,s;j}
\rightarrow \mathcal{L}_{r,s;j}
\rightarrow 0.$$
\end{theom}

The proof is similar to that of Theorem \ref{chBGG2}, and we omit it.

\begin{rem}\label{Main1ex}
As a consequence of \cite[Theorem 7.1 and 7.2]{Ad99}, simple 
$L_{c}$-modules are exhasted by the following (see \cite[Theorem 2.1 and Lemma 4.1]{sato2017modular} for detail):
\begin{enumerate}
\item $\mathcal{L}(r,s)^{\theta}$ for $(r,s)\in\mathcal{I}_{\sf KW}$ and $\theta\in\mathbb{Z}$,
\item $\mathcal{L}_{r,s;j}$ for $(r,s)\in\mathcal{I}_{\sf BPZ}$ and $j\notin\{j_{r,s},j_{p-r,p'-s}\}+\mathbb{Z}$.
\end{enumerate}
Since we obtain BGG type resolution for $\mathcal{L}(r,s)^{\theta}$
by applying the exact functor $(?)^{\theta}$ (see Appendix \ref{SF}) to the above exact sequence (\ref{tBGG}),
we conclude Theorem \ref{Main1}.
\end{rem}

\section{Uniqueness of singular generator}\label{uSV}

\subsection{Affine VOA side}

We recall the following singular vector formula for the generalized Verma module $V(r)$
at the Kac--Wakimoto admissible level $k=-2+a^{-1}$,
which is a direct corollary of \cite[Proposition 2.1]{MFF86} and Theorem \ref{gBGG}.

\begin{prp}\label{MFF}
The maximum proper $\widehat{\mf{sl}}_{2}$-submodule of 
$V(r)$ is generated by
the Malikov--Feigin--Fuchs (MFF) singular vector
\begin{align*}
&v_{p,p'}(r):=E_{-1}^{(2p'-1)\kappa-r}
F_{0}^{(2p'-2)\kappa-r}\cdots F_{0}^{2\kappa-r}
E_{-1}^{\kappa-r}\ket{r}
\in V(r)
\end{align*}
in the $\bigl(L_{0}^{\sf Sug},H_{0}\bigr)$-eigenspace with eigenvalue 
$\bigl((p-r)p',2p-r-1\bigr),$
where $\kappa:=a^{-1}=\frac{p}{\,p'}$ and $\ket{r}$ is the highest weight vector in $V(r)$.
\end{prp}

\begin{ex}
When $r=1$, we obtain the following (see \cite[Theorem 3.3]{AM95}):
the maximum proper ideal of the universal affine vertex algebra $V_{k}(\mf{sl}_{2})\simeq V(1)$ is generated by the MFF singular vector $v_{p,p'}:=v_{p,p'}(1)\in V_{k}(\mf{sl}_{2})$.
\end{ex}

The next lemma plays a key role in this section.

\begin{lem}\label{affunique}
There exists a unique $\widehat{\mf{sl}}_{2}$-module homomorphism
$$f_{p,p'}(r)\colon R(r)\to V(r);\ 
v(r)\mapsto F_{0}^{p-r}v_{p,p'}(r),$$
where $R(r):=R_{(p-r)p'+\Delta(j_{r,0}),j_{r,0},k}$
and $v(r)$ is the relaxed highest weight vector of $R(r)$.
Moreover we have
\begin{equation*}
\Hom_{\widehat{\mf{sl}}_{2}}\!\!\big(R(r),V(r)\big)=\mathbb{C}f_{p,p'}(r).
\end{equation*}
\end{lem}

\begin{proof}
Since $v_{p,p'}(r)$ is a singular vector, the vector $F_{0}^{p-r}v_{p,p'}(r)$
satisfies the same annihilation relations as those for $v(r)$.
Thus the former half follows.

Let $f\colon R(r)\to V(r)$ be an $\widehat{\mf{sl}}_{2}$-module homomorphism.
Then the vector $f\big(v(r)\big)$ lies in 
the $\bigl(L_{0}^{\sf Sug},H_{0}\bigr)$-eigenspace with eigenvalue 
$\bigl((p-r)p',r-1\bigr)$ in the maximum proper $\widehat{\mf{sl}}_{2}$-submodule 
of $V(r)$.
By the Poincar\'{e}--Birkoff--Witt (PBW) theorem, the eigenspace is $\mathbb{C}F_{0}^{p-r}v_{p,p'}(r)$.
This proves the latter half.
\end{proof}

\subsection{$\mathcal{N}=2$ VOSA side}

In this subsection, the letter $c$ always stands for 
the central charge $c_{p,p'}=3(1-2a)$.
The next proposition gives a generator of the maximum proper submodule of $\mathcal{V}(r)$.

\begin{theom}\label{Uniqueness}
Set $\mathcal{M}(r):=\mathcal{M}_{(p-r)p'+aj_{r,0},2aj_{r,0},c}$.
Then we have
$${\sf dim}\Hom_{\mf{ns}_{2}}\!\!\big(\mathcal{M}(r),\mathcal{V}(r)\big)=1,$$
that is, there exists a unique singular vector 
$\mathcal{N}_{p,p'}(r)\in \mathcal{V}(r)^{\bar{0}}$ with $(L_{0},J_{0})$-eigenvalue 
$\bigl((p-r)p'+aj_{r,0},2aj_{r,0}\bigr)$
up to non-zero scalar multiple.
In addition, the maximum proper submodule of $\mathcal{V}(r)$ is generated by 
$\mathcal{N}_{p,p'}(r)$.
\end{theom}

\begin{proof}
Since the functor $\Omega^{+}_{j_{r,0}}$ is fully faithful by \cite[Theorem 4.4]{sato2016equivalences},
we have
$$\Hom_{\mf{ns}_{2}}\!\!\big(\mathcal{M}(r),\mathcal{V}(r)\big)=\mathbb{C}g_{p,p'}(r),$$
where $g_{p,p'}(r):=\Omega^{+}_{j_{r,0}}\big(f_{p,p'}(r)\big)$.
The exact functor $\Omega^{+}_{j_{r,0}}$ sends
the exact sequence
$$R(r)\xrightarrow{f_{p,p'}(r)}
V(r)\overset{\pi}{\longrightarrow}L(r,0)$$
to the following exact sequence
$$\mathcal{M}(r)\xrightarrow{g_{p,p'}(r)}
\mathcal{V}(r)\overset{\pi'}{\longrightarrow}\mathcal{L}(r,0),$$
where $\pi$ and $\pi'$ are the natural projections.
This completes the proof.
\end{proof}

\begin{ex}\label{maxiproper}
When $r=1$, we have
$${\sf dim}\Hom_{\mf{ns}_{2}}\!\!\left(\mathcal{M}_{(p-1)p',0,c},V_{c}\right)=1,$$
that is, there exists a unique singular vector 
$\mathcal{N}_{p,p'}:=\mathcal{N}_{p,p'}(1)\in V_{c}^{\bar{0}}$ with $(L_{0},J_{0})$-eigenvalue $\bigl((p-1)p',0\bigr)$
up to non-zero scalar multiple.
In addition, the maximum proper ideal $I_{c}$ of
the vertex operator superalgebra $V_{c}$ is generated by $\mathcal{N}_{p,p'}$.
\end{ex}

\begin{rem}\label{EG1}
When $p'=1$, the simple highest weight module $\mathcal{L}(r,0)$ lies in the $\mathcal{N}=2$ unitary minimal series \cite{boucher1986determinant} and 
$\mathcal{V}(r)$ contains no subsingular vectors by \cite[Theorem 4]{dorrzapf1998embedding}.
Therefore Theorem \ref{Uniqueness} for $p'=1$ is a direct corollary of \cite[Theorem 1]{dorrzapf1998embedding}.
On the other hand, when $p'>1$, it has not been known that $\mathcal{V}(r)$
contains no subsingular vectors (cf.\,\cite[Remark 8]{eholzer1997unitarity}).
\end{rem}

\section{Applications}\label{app}

\subsection{Application 1: Structure of Zhu's algebra}

In this subsection, we compute the twisted Zhu algebras of $L_{c}$
by using the singular vector $\mathcal{N}_{p,p'}$ in $V_{c}$.
See \ref{DefZhuSec} for the definition of the twisted Zhu algebras.

\subsubsection{The $\sigma$-twisted (=\,Ramond) case}

Take a $\mathbb{Z}_{2}$-homogeneous $\mathbb{C}$-basis of
the general linear Lie superalgebra $\mf{gl}_{1|1}=\End_{\mathbb{C}}(\mathbb{C}^{1|1})$ as
$$
Z:=\frac{1}{2}\begin{pmatrix}
1 & 0\\
0 & 1
\end{pmatrix},\ 
J:=\frac{1}{2}
\begin{pmatrix}
1 & 0\\
0 & -1
\end{pmatrix},\ 
\Psi^{+}:=
\begin{pmatrix}
0 & 1\\
0 & 0
\end{pmatrix},\ 
\Psi^{-}:=
\begin{pmatrix}
0 & 0\\
1 & 0
\end{pmatrix}.
$$
By direct computations, one can verify the following:

\begin{prp}\label{Prp41}
For $c\in\mathbb{C}$, there exists a unique superalgebra isomorphism
\begin{equation*}
i_{\sigma}\colon U(\mf{gl}_{1|1})\overset{\simeq}{\longrightarrow}
A_{\sigma}(V_{c})
\end{equation*}
such that 
$i_{\sigma}(Z)=[\mathbf{L}]-\frac{c}{24}[\mathbf{1}^{c}],\ 
i_{\sigma}(J)=[\mathbf{J}],$ and
$i_{\sigma}(\Psi^{\pm})=[\mathbf{G}^{\pm}].$ 
\end{prp}

Then we obtain the following description of $A_{\sigma}(L_{c})$.

\begin{theom}\label{Kernel1}
When $c=c_{p,p'}$, the kernel of the natural projection
$$\pi_{\sigma}\colon A_{\sigma}(V_{c})\to A_{\sigma}(L_{c})$$
is generated by the coset $[\mathcal{N}_{p,p'}]$ of the singular vector 
$\mathcal{N}_{p,p'}\in V_{c}^{\bar{0}}$.
\end{theom}

\begin{proof}
Let $\mathcal{I}_{c,\sigma}$ be the two-sided ideal of $A_{\sigma}(V_{c})$
which is generated by $[\mathcal{N}_{p,p'}]$.
Since $\mathcal{I}_{c,\sigma}\subseteq\Ker\pi_{\sigma}=I_{c}+O_{\sigma}(V_{c})$ is clear, it sufficies to prove that
$\mathcal{I}_{c,\sigma}\supseteq I_{c}+O_{\sigma}(V_{c})$.
Since $I_{c}=U\big((\mf{ns}_{2})_{-}\big)\mathcal{N}_{p,p'}$ by Example \ref{maxiproper},
it is proved by a super analog of \cite[Lemma 2.1.2]{zhu1996modular} (see e.g.\,\cite[Lemma 3.1 (i)]{dong2006twisted}) and by induction
with respect to the PBW filtration of $U\big((\mf{ns}_{2})_{-}\big)$.
\end{proof}

Then Theorem \ref{Main3} (1) follows from Proposition \ref{Prp41} and Theorem \ref{Kernel1}.

\begin{ex}
By the PBW theorem, there exist $P_{1}, P_{2}\in \mathbb{C}[x,y]$ such that
$$\phi_{c}:=i_{\sigma}^{-1}\big([\mathcal{N}_{p,p'}]\big)=P_{1}(Z,J)+P_{2}(Z,J)\Psi^{-}\Psi^{+}
\in U(\mf{gl}_{1|1})^{\bar{0}}.$$
We give some examples of $\phi_{c}$ explicitly.
For brevity, we write $v\propto w$ if $v$ and $w$ are proportional.
\begin{enumerate}
\item
When $(p,p')=(4,1)$, we have $c_{4,1}=\frac{3}{2}$ and
\begin{align*}
\mathcal{N}_{4,1}&\propto\Biggl(10J_{-3}-3L_{-3}
+3G^{+}_{-\frac{3}{2}}G^{-}_{-\frac{3}{2}}
-12L_{-2}J_{-1}
+8J_{-1}^{3}
\Biggr)\mathbf{1}^{c}
\end{align*}
(see \cite[p.72]{eholzer1997unitarity}).
By some computations,
we obtain
$$\phi_{\frac{3}{2}}\propto (4J-1)\bigl(J(4J+1)-6Z\bigr)-6\Psi^{-}\Psi^{+}.$$
\item
When $(p,p')=(2,3)$, we have $c_{2,3}=-6$ and
\begin{align*}
\mathcal{N}_{2,3}&\propto\Biggl(-10J_{-3}-6L_{-3}
+6G^{+}_{-\frac{3}{2}}G^{-}_{-\frac{3}{2}}
+6L_{-2}J_{-1}
+J_{-1}^{3}
\Biggr)\mathbf{1}^{c}
\end{align*}
(see \cite[p.72]{eholzer1997unitarity}).
By some computations,
we obtain
$$\phi_{-6}\propto (J+1)\bigl(J(J-1)+6Z\bigr)-6\Psi^{-}\Psi^{+}$$
\item
When $(p,p')=(3,2)$, we have $c_{3,2}=-1$ and
\begin{align*}
\mathcal{N}_{3,2}\propto\Biggl(
&
42J_{-4}
+24L_{-4}
+27J_{-2}J_{-2}
-84J_{-3}J_{-1}\\
&
-6G^{+}_{-\frac{3}{2}}G^{-}_{-\frac{5}{2}}
+6G^{+}_{-\frac{5}{2}}G^{-}_{-\frac{3}{2}}
-32L_{-2}L_{-2}\\
&
-36L_{-3}J_{-1}
+36J_{-1}G^{+}_{-\frac{3}{2}}G^{-}_{-\frac{3}{2}}
+12L_{-2}J_{-1}^{2}
+9J_{-1}^{4}
\Biggr)\mathbf{1}^{c}
\end{align*}
(see \cite[p.72]{eholzer1997unitarity}).
By some computations,
we obtain
\begin{align*}\phi_{-1}\propto
\big((6J+1)(6J+5)-48Z\big)\big((6J-1)(6J-5)+96Z\big)-72^{2}J\Psi^{-}\Psi^{+}.
\end{align*}
\end{enumerate}
\end{ex}


In the same way as \cite[Theorem 7.1 and 7.2]{Ad99},
we obtain the classification of simple $\sigma$-twisted $L_{c}$-modules
by the Kazama--Suzuki coset construction\footnote{We can also obtain the classification by using the equivalence of categories in Appendix \ref{SF}.}.
Then, by a generalization of \cite[Theorem 2.2.2]{zhu1996modular}
(see e.g.\,\cite[Corollary 5.1.8]{xu1998introduction}, \cite[Theorem 6.5]{dong2006twisted}), we also obtain 
the classification of finite-dimensional simple $A_{\sigma}(L_{c})$-modules as follows.

\begin{prp}\label{classification}
Let $c=c_{p,p'}$.
For $(z,j)\in\mathbb{C}^{2}$, let $\mathcal{L}_{z,j}$ be the simple highest weight $\mf{gl}_{1|1}$-module (with respect to the standard Borel subalgebra $\mf{b}=\mathbb{C}Z\oplus\mathbb{C}J\oplus\mathbb{C}\Psi^{+}$)
of highest weight $zZ^{*}+jJ^{*}$, where $(Z^{*},J^{*})$ is the dual basis of $(Z,J)$.
Then the action of $U(\mf{gl}_{1|1})$ on $\mathcal{L}_{z,j}$ factors through 
that of $A_{\sigma}(L_{c})$ if and only if
the pair $(z,j)$ lies in the disjoint union of the following sets:
\begin{align}
&\label{P01}\left\{\left(z_{r,\theta},\,\mu^{\sf}_{r,\theta}+\frac{1}{2}\right)
\,\Bigg|\,(r,0)\in\mathcal{I}_{\sf KW},\,0\leq\theta\leq r-1,\,\theta\in\mathbb{Z}\right\},\\
&\label{P02}\left\{\left(\frac{(ar-s)^{2}-\mu^{2}}{4a},\,
\mu+\frac{1}{2}\right)\,\Bigg|\,(r,s)\in\mathcal{I}_{\sf BPZ},\,\mu\in\mathbb{C}\right\},
\end{align}
where $z_{r,\theta}:=a(\theta+1)(r-1-\theta)$ and $\mu^{\sf}_{r,\theta}:=a\big((r-1-\theta)-(\theta+1)\big)$.
\end{prp}

\begin{rem}
When $p'=1$, the set (\ref{P01}) corresponds to 
the class $\mathrm{P}^{-}_{0}$ unitary minimal series
of central charge $\tilde{c}=\frac{c}{3}$ in \cite{boucher1986determinant} and the set (\ref{P02}) is empty.
\end{rem}

As a consequence, we obtain the following.

\begin{cor}\label{finiteness}
For $c=c_{p,p'}$, the following conditions are all equivalent:
\begin{enumerate}
\item $p'=1$,
\item $A_{\sigma}(L_{c})$ is finite-dimensional,
\item every finite-dimensional $A_{\sigma}(L_{c})$-module is completely reducible.
\end{enumerate}
\end{cor}

\begin{proof}
First the implication $(1)\Rightarrow(2)$ follows from 
the $\sigma$-twisted regularity of $L_{c}$ (cf.\,\cite[Theorem 8.1]{Ad01})
and a natural generalization of \cite[Theorem 3.8]{li1999some}.
Next the implication $(2)\Rightarrow(3)$ follows from Proposition \ref{classification}.
At last we prove that $(3)$ implies $(1)$.
We suppose that $p'\neq1$.
Let $\mathcal{M}_{z,j}$ be the Verma module of $\mf{gl}_{1|1}$ of highest weight 
$(z,j)\in\mathbb{C}^{2}$.
By Proposition \ref{classification} and
the existence of a non-split exact sequence $0\to\mathcal{L}_{0,j-1}\to\mathcal{M}_{0,j}\to\mathcal{L}_{0,j}\to0$
of $\mf{gl}_{1|1}$-modules, the element $\phi_{c}$ acts trivially on 
the indecomposable non-simple $\mf{gl}_{1|1}$-module $\mathcal{M}_{0,\mu_{r,r-1}+\frac{1}{2}}$. Thus $(3)$
implies $(1)$.
\end{proof}

\subsubsection{The $\id$-twisted (=\,Neveu--Schwarz) case}\label{NScase}

The following isomorphism seems to be well known (e.g.\,\cite[Remark 1.1]{Ad99})
and is proved in the same way as \cite[Lemma 3.1]{kac1994vertex}.

\begin{prp}\label{zhu}
For $c\in\mathbb{C}$, there exists a unique purely even $\mathbb{Z}_{2}$-graded algebra isomomorphism
$$i_{\id}\colon\mathbb{C}[{\sf h},{\sf q}]\overset{\simeq}{\longrightarrow} A_{\id}(V_{c})$$
such that $i_{\id}({\sf h})=[\mathbf{L}]$ and $i_{\id}({\sf q})=[\mathbf{J}].$
\end{prp}

Then we obtain a similar result on the structure of $A_{\id}(L_{c})$.

\begin{theom}
When $c=c_{p,p'}$, the kernel of the natural projection
$$\pi_{\id}\colon A_{\id}(V_{c})\to A_{\id}(L_{c})$$
is generated by the cosets
$[\mathcal{N}_{p,p'}]$
and $\bigl[G^{+}_{-\frac{1}{2}}G^{-}_{-\frac{1}{2}}\mathcal{N}_{p,p'}\bigr]$
in $A_{\id}(V_{c})$.
\end{theom}

\begin{proof}
Let $\mathcal{I}_{c,\id}$ be the ideal of $A_{\id}(V_{c})$ generated by
$[\mathcal{N}_{p,p'}]$ and $\bigl[G^{+}_{-\frac{1}{2}}G^{-}_{-\frac{1}{2}}\mathcal{N}_{p,p'}\bigr]$.
We only need to verify $I_{c}+O_{\id}
\subseteq\mathcal{I}_{c,\id}.$
Since the even singular vector $\mathcal{N}_{p,p'}$ generates the maximum proper submodule 
of $V_{c}$ by Example \ref{maxiproper},
it suffices to prove that
$U_{-}\mathcal{N}_{p,p'}+O_{\id}\subseteq\mathcal{I}_{c,\id}$ and 
$G^{+}_{-\frac{1}{2}}G^{-}_{-\frac{1}{2}}
U_{-}\mathcal{N}_{p,p'}+O_{\id}\subseteq\mathcal{I}_{c,\id}$, where
$U_{-}:=U\bigl((\mf{ns}_{2})_{-}^{\bar{0}}\bigr)$.
The former follows from a direct computation and
the latter is proved by induction with respect to the PBW filtration of $U_{-}$.
\end{proof}

\begin{cor}[\cite{eholzer1997unitarity}]\label{Kernel2}
The kernel of the mapping
$$\pi_{\id}\circ i_{\id}\colon \mathbb{C}[{\sf h},{\sf q}]\xrightarrow{\simeq}
A_{\id}(V_{c})
\to A_{\id}(L_{c})$$
generated by the two polynomials
$$f_{c}:=i_{\id}^{-1}\bigl([\mathcal{N}_{p,p'}]\bigr)
,\ g_{c}:=i_{\id}^{-1}\bigl([G^{+}_{-\frac{1}{2}}G^{-}_{-\frac{1}{2}}\mathcal{N}_{p,p'}]\bigr)\in\mathbb{C}[{\sf h},{\sf q}].$$
\end{cor}

\begin{rem}\label{EG2}
In \cite[\S3]{eholzer1997unitarity}, W.~Eholzer and M.R.~Gaberdiel 
discussed the relation between the structure of the Zhu algebra 
$A_{\id}(L_{c})$ and
the existence of an ``additional independent bosonic singular vector'' $\mathcal{N}$ in 
$V_{c}$.
By the uniqueness, their singular vector $\mathcal{N}$
coincides with $\mathcal{N}_{p,p'}$ (up to non-zero scalar multiple).
\end{rem}

By standard computations (see e.g.\,\cite[Lemma 1.1]{kac1994vertex}),
one can verify that 
the pair $(f_{c},g_{c})$ coincides with $(p_{1},p_{2})$ appeared in \cite[Table 3.1]{eholzer1997unitarity}.

\subsection{Application 2: Structure of Frenkel--Zhu's bimodule}

In this subsection, we compute the Frenkel--Zhu bimodule structure on
a certain simple module over the simple $\mathcal{N}=2$ vertex operator superalgebra of central charge $c=c_{3,2}=-1$.

We first compute the $\mathbb{Z}_{2}$-graded $A_{\id}(V_{c})$-bimodule structure of
 $A_{\id}\big(\mathcal{M}_{j,c}^{+}\big).$
We always identify $A_{\id}(V_{c})$ with $\mathbb{C}[{\sf h},{\sf q}]$ by the isomorphism $i_{\sf id}$
in Proposition \ref{zhu}.

\begin{lem}\label{FZchiral}
Set a $\mathbb{Z}_{2}$-graded structure on 
$\mathbb{C}[x_{\ell},{x}_{r},{y}]\oplus\mathbb{C}[x_{\ell},{x}_{r},{y}]\psi$
by ${\sf deg}\,f=\bar{0}$ and ${\sf deg}\,f\psi=\bar{1}$ for $f\in\mathbb{C}[x_{\ell},{x}_{r},{y}]$.
For $j\in\mathbb{C}$, we define a $\mathbb{Z}_{2}$-graded $\mathbb{C}[{\sf h},{\sf q}]$-bimodule structure on
$\mathbb{C}[x_{\ell},{x}_{r},{y}]\oplus\mathbb{C}[x_{\ell},{x}_{r},{y}]\psi$
by
\begin{align*}
&{\sf h}.(f+g\psi):=x_{\ell}(f+g\psi),\ {\sf q}.(f+g\psi):=
(y+j)f+(y+j-1)g\psi,\\
&(f+g\psi).{\sf h}:=x_{r}(f+g\psi),\ (f+g\psi).{\sf q}:=y(f+g\psi)
\end{align*}
for $f,g\in\mathbb{C}[x_{\ell},{x}_{r},{y}]$.
Then there exists a unique $\mathbb{Z}_{2}$-graded $\mathbb{C}[{\sf h},{\sf q}]$-bimodule isomorphism
$$\mathbb{C}[x_{\ell},{x}_{r},{y}]\oplus\mathbb{C}[x_{\ell},{x}_{r},{y}]\psi
\to A_{\id}\big(\mathcal{M}_{j,c}^{+}\big)$$
such that
$1\mapsto\big[\ket{j,c}^{\sf sc}\big]$ and $\psi\mapsto\big[G_{-\frac{1}{2}}^{-}\ket{j,c}^{\sf sc}\big].$
\end{lem}

\begin{proof}
The existence and uniqueness of such a bimodule homomorphism are easily verified by calculation.
By some computations (cf.\,\cite[\S5]{wang1993rationality},\,\cite[Lemma 3.1]{kac1994vertex}),
one can prove that $\mathcal{O}_{\id}\big(\mathcal{M}_{j,c}^{+}\big)$ is linearly spanned by
vectors of the form
$(\mathbf{L}_{(-\ell-2)}+2\mathbf{L}_{(-\ell-1)}+\mathbf{L}_{(-\ell)})v$,
$(\mathbf{J}_{(-\ell-2)}+\mathbf{J}_{(-\ell-1)})v$,
or $(\mathbf{G}^{\pm}_{(-\ell-1)}+\mathbf{G}^{\pm}_{(-\ell)})v$
for $\ell\in\mathbb{Z}_{\geq0}$ and $v\in\mathcal{M}_{j,c}^{+}$.
Then, by the PBW theorem, we can explicitly construct the inverse mapping of 
the induced linear mapping
between the associated graded vector spaces (note that we consider the left-hand side is canonically identified with its associated graded space).
This proves the bijectivity of the bimodule homomorphism.
\end{proof}

Next, as a corollary of the classification in \cite[Theorem 7.2]{Ad99},
the set of isomorphism classes of $\mathbb{Z}_{2}$-graded simple $L_{-1}$-modules is given by
\begin{align*}
\left\{\left.
\mathcal{L}_{\frac{\epsilon^{2}}{3},\frac{2\epsilon}{3},-1},\,
\Pi\mathcal{L}_{\frac{\epsilon^{2}}{3},\frac{2\epsilon}{3},-1}\right|
\epsilon\in\{-1,0,1\}\right\}
\sqcup\left\{\left.
\mathcal{L}_{-\frac{1+3j^{2}}{8},j,-1},\,
\Pi\mathcal{L}_{-\frac{1+3j^{2}}{8},j,-1}\right|j\in\mathbb{C}\right\},
\end{align*}
where $\Pi$ is the parity changing functor on the category of $\mathbb{Z}_{2}$-graded vector spaces.
In what follows, we denote the corresponding $\mathbb{Z}_{2}$-graded 
simple left $A_{\sf id}(L_{-1})$-modules (cf.\,\cite[Theorem 1.3]{kac1994vertex}) by
$$\big\{\,
\mathbb{C}(\epsilon),\,\Pi\mathbb{C}(\epsilon)\,\big|\,
\epsilon\in\{-1,0,1\}\big\}
\sqcup\big\{\mathbb{C}_{j},\,\Pi\mathbb{C}_{j}\,\big|\,j\in\mathbb{C}\,\big\}.$$
Note that we have 
$\mathcal{L}(1,0)\simeq\mathcal{L}_{0,0,-1}$, $\mathcal{L}(2,0)\simeq\mathcal{L}_{\frac{1}{3},\frac{2}{3},-1}$,
and $\mathcal{L}(2,0)^{1}\simeq\mathcal{L}_{\frac{1}{3},-\frac{2}{3},-1}$.
Then, by using the isomorphism in Lemma \ref{FZchiral}, we can describe the structure of $A_{\id}\big(\mathcal{L}_{\frac{1}{3},\frac{2}{3},-1}\big)$ as follows.

\begin{lem}\label{FZbim}
The kernel of the natural surjective $\mathbb{C}[{\sf h},{\sf q}]$-bimodule homomorphism
$$\mathbb{C}[x_{\ell},{x}_{r},{y}]\oplus\mathbb{C}[x_{\ell},{x}_{r},{y}]\psi
\xrightarrow{\simeq}A_{\id}\big(\mathcal{M}_{\frac{2}{3},-1}^{+}\big)
\to A_{\id}\big(\mathcal{L}_{\frac{1}{3},\frac{2}{3},-1}\big)$$
is generated by
\begin{align*}
&f_{1}:=3P(P+R)-Q,\\
&f_{2}:=(2P+R)(P-R),\\
&f_{3}:=(2P+R)\big((3P-4)(P-R)-3Q+2\big)\\
&g_{1}\psi:=(2P+2R-1)\psi,\\
&g_{2}\psi:=(Q-P-R)\psi,\\
&g_{3}\psi:=\big(4P^{2}+(1-2R)P-2R^{2}+2R-3Q\big)\psi,\\
\end{align*}
where $P:=x_{\ell}-x_{r}$, $Q:=x_{\ell}+x_{r}+\frac{1}{3}$, and $R:=y+\frac{1}{3}$.
\end{lem}

\begin{proof}
Set $w_{1}:=G_{-\frac{3}{2}}^{-}G_{-\frac{1}{2}}^{-}\big|2/3,-1\big\rangle^{\sf sc}$ and
\begin{align*}
&w_{2}:=
(4J_{-2}-3G_{-\frac{3}{2}}^{+}G_{-\frac{1}{2}}^{-}-2L_{-1}J_{-1}-2J_{-1}^{2}+4L_{-1}^{2})
\big|2/3,-1\big\rangle^{\sf sc}.
\end{align*}
By direct calculation, one can verify
$w_{2}$ is a singular vector in $\mathcal{M}_{\frac{2}{3},-1}^{+}$,
which is a lift of $\mathcal{N}_{3,2}(2)$ with respect to
the natural projection from $\mathcal{M}_{\frac{2}{3},-1}^{+}$ to $\mathcal{V}(2)$.
By Lemma \ref{weyl} and Theorem \ref{Uniqueness}, 
$w_{1}$ and $w_{2}$ generate the maximum proper submodule of $\mathcal{M}_{\frac{2}{3},-1}^{+}$.
Then, by \cite[Proposition 1.2]{kac1994vertex} and 
the PBW theorem, the kernel of
the natural projection from $A_{\id}\big(\mathcal{M}_{\frac{2}{3},-1}^{+}\big)$
to $A_{\id}\big(\mathcal{L}_{\frac{1}{3},\frac{2}{3},-1}\big)$
is generated by
$f_{1}':=\big[G_{-\frac{1}{2}}^{+}G_{\frac{1}{2}}^{+}w_{1}\big]$,
$f_{2}':=[w_{2}]$, $f_{3}':=\big[G_{-\frac{1}{2}}^{+}G_{-\frac{1}{2}}^{-}w_{2}\big]$,
$g_{1}'\psi:=\big[G_{\frac{1}{2}}^{+}w_{1}\big]$, 
$g_{2}'\psi:=\big[G_{-\frac{1}{2}}^{+}w_{1}\big]$,
and $g_{3}'\psi:=\big[G_{-\frac{1}{2}}^{-}w_{2}\big]$.
By some computations, one can verify that $f_{i}\propto f_{i}'$
and $g_{i}\psi\propto g_{i}'\psi$ for $i\in\{1,2,3\}$.
This completes the proof.
\end{proof}

At last, by Lemma \ref{FZbim} and direct calculation, we obtain the following.

\begin{prp}\label{FR}
For $\epsilon\in\{-1,0,1\}$ and $j\in\mathbb{C}$, we have
\begin{align*}
&A_{\id}\big(\mathcal{L}_{\frac{1}{3},\frac{2}{3},-1}\big)\otimes_{A_{\id}(L_{-1})}
\mathbb{C}(\epsilon)
\simeq 
\begin{cases}
\mathbb{C}(0)
 & \text{ if }\epsilon=-1,\\
\mathbb{C}(1)
 & \text{ if }\epsilon=0,\\
\Pi\mathbb{C}_{\frac{1}{3}} & \text{ if }\epsilon=1,
\end{cases}\\
&A_{\id}\big(\mathcal{L}_{\frac{1}{3},\frac{2}{3},-1}\big)\otimes_{A_{\id}(L_{-1})}
\mathbb{C}_{j}
\simeq 
\begin{cases}
\displaystyle
\mathbb{C}_{\frac{1}{3}}
\oplus
\Pi\mathbb{C}(-1) & \text{ if } j=-\frac{1}{3},\\
\mathbb{C}_{j+\frac{2}{3}} & \text{ if } j\neq-\frac{1}{3}.
\end{cases}
\end{align*}
\end{prp}

\begin{rem}\label{conjFR}
In \cite[Example 5.2]{sato2017modular}, the second author conjectures that
the dimension of the space of $\mathbb{Z}_{2}$-graded intertwining operators 
(called the fusion rule) of type
$$\left(
\begin{array}{c}
\mathcal{L}_{-\frac{1}{8}(1+3(j+\frac{2}{3})^{2}),j+\frac{2}{3},-1}\\
\mathcal{L}_{\frac{1}{3},\frac{2}{3},-1}\ \ \mathcal{L}_{-\frac{1}{8}(1+3j^{2}),j,-1}
\end{array}
\right)$$
is equal to $1$ if $j\notin\{\frac{1}{3}\}+\frac{2}{3}\mathbb{Z}$.
By a super analog of \cite[Proposition 2.10]{li1999determining} (e.g.\,\cite[Theorem 1.5 (1)]{kac1994vertex}) and Proposition \ref{FR}, the conjectural fusion rule is actually bounded by
$${\sf dim}_{\mathbb{C}}\Hom_{A_{\id}(L_{-1})}\Big(
A_{\id}\big(\mathcal{L}_{\frac{1}{3},\frac{2}{3},-1}\big)\otimes_{A_{\id}(L_{-1})}
\mathbb{C}_{j},\mathbb{C}_{j+\frac{2}{3}}\oplus\Pi\mathbb{C}_{j+\frac{2}{3}}\Big)=1.$$
Based on an appropriate super analog of \cite[Theorem 2.11]{li1999determining},
if the simple $L_{-1}$-module $\mathcal{L}_{-\frac{1}{8}(1+3j^{2}),j,-1}$ is a generalized Verma
weak $L_{-1}$-module (in the sense of \cite[Definition 2.7]{li1999determining})
for $j\notin\{\frac{1}{3}\}+\frac{2}{3}\mathbb{Z}$,
then the conjectural fusion rule is true.
At this moment, however, we do not know how to prove that.
\end{rem}

\appendix

\section{$\mathcal{N}=2$ vertex operator superalgebra}

\subsection{$\mathcal{N}=2$ vertex operator superalgebras}\label{N2VOSA}

For $c\in\mathbb{C}$, the vacuum $\mf{ns}_{2}$-module $V_{c}$ of central charge $c$
is the $\mathbb{Z}_{2}$-graded $\mf{ns}_{2}$-module freely generated by
an even vector ${\bf 1}^{c}$ subject to the relations
\begin{align*}
&(\mf{ns}_{2})_{+}{\bf 1}^{c}:=\{0\},\ G_{-\frac{1}{2}}^{+}{\bf 1}^{c}
=G_{-\frac{1}{2}}^{-}{\bf 1}^{c}:=0,\\
&L_{0}{\bf 1}^{c}:=0,\ J_{0}{\bf 1}^{c}:=0,\ C{\bf 1}^{c}:=c{\bf 1}^{c}.
\end{align*}

\begin{prp}[{\cite[Proposition 1.1]{Ad99}}]
There exists a unique vertex superalgebra
$(V_{c},Y,{\bf 1}^{c})$ which is strongly generated by the following mutually local fields
\begin{align*}
Y(\mathbf{L};z)\ &=\sum_{n\in\mathbb{Z}}\ \mathbf{L}_{(n)}z^{-n-1}
\;:=\sum_{n\in\mathbb{Z}}L_{n-1}z^{-n-1},\\
Y(\mathbf{G}^{\pm};z)&
=\sum_{n\in\mathbb{Z}}\ \mathbf{G}^{\pm}_{(n)}z^{-n-1}
:=\sum_{n\in\mathbb{Z}}G^{\pm}_{n-\frac{1}{2}}z^{-n-1},\\
Y(\mathbf{J};z)\ &
=\sum_{n\in\mathbb{Z}}\ \mathbf{J}_{(n)}z^{-n-1}
\;:=\sum_{n\in\mathbb{Z}}J_{n}z^{-n-1},
\end{align*}
where 
$\mathbf{L}:=L_{-2}{\bf 1}^{c}$, $\mathbf{G}^{\pm}:=G^{\pm}_{-\frac{3}{2}}{\bf 1}^{c}$,
and $\mathbf{J}:=J_{-1}{\bf 1}^{c}$.
In addition, the vertex superalgebra $(V_{c},Y,{\bf 1}^{c})$
together with $\mathbf{L}$ as a conformal vector 
forms a vertex operator superalgebra of central charge $c$.
\end{prp}

When $c\neq0$, we write
$L_{c}$ for its simple quotient vertex operator superalgebra.

\subsection{Spectral flow twists}\label{SF}

In this subsection $V$ stands for $V_{c}$ or $L_{c}$.
Let $\varepsilon,\varepsilon'\in\big\{0,\frac{1}{2}\big\}$ 
and let $(\mathcal{M},Y_{\mathcal{M}})$ be a $(1-\varepsilon)\mathbb{Z}_{\geq0}$-gradable 
$\mathbb{Z}_{2}$-graded $\sigma^{1-2\varepsilon}$-twisted $V$-module.
The next lemma is proved in a similar way of \cite[Proposition 3.2]{li1997physics}
(see \cite[Theorem 3.3.8]{xu1998introduction} for detail).

\begin{lem}\label{SFtw}
For $\theta\in\mathbb{Z}+\varepsilon'$, we define
$$\Delta(\theta\mathbf{J};z):=z^{\theta J_{0}}{\sf exp}\Bigl(\sum_{\ell=1}^{\infty}
\frac{\theta J_{\ell}}{-\ell}(-z)^{-\ell}\Bigr)\in\End(V)[\![z^{\pm(1-\varepsilon')}]\!].$$
Then the mapping
\begin{equation*}
Y_{\mathcal{M}^{\theta}}(?;z)
:=Y_{\mathcal{M}}(\Delta(\theta\mathbf{J};z)?;z):\,V\rightarrow\End(\mathcal{M})[\![z^{\pm(\frac{1}{2}+|\varepsilon-\varepsilon'|)}]\!]
\end{equation*}
gives rise to a $(1-|\varepsilon-\varepsilon'|)\mathbb{Z}_{\geq0}$-gradable 
$\mathbb{Z}_{2}$-graded $\sigma^{1-2|\varepsilon-\varepsilon'|}$-twisted $V$-module
$\mathcal{M}^{\theta}:=(\mathcal{M},Y_{\mathcal{M}^{\theta}})$.
\end{lem}

We call $\mathcal{M}^{\theta}$ the spectral flow twisted module of $\mathcal{M}$.
Note that the spectral flow twisted module of a highest weight module is not 
a highest weight module in general.

Let $\mathcal{C}^{\varepsilon}$ be the category of $(1-\varepsilon)\mathbb{Z}_{\geq0}$-gradable
$\sigma^{1-2\varepsilon}$-twisted $V$-modules.
Then it is clear that 
the assignments $\mathcal{M}\mapsto\mathcal{M}^{\theta}$
and $\Hom(\mathcal{M}_{1},\mathcal{M}_{2})\to\Hom(\mathcal{M}_{1}^{\theta},\mathcal{M}_{2}^{\theta});\,f\mapsto f$
give rise to an equivalence of categories
$(?)^{\theta}\colon\mathcal{C}^{\varepsilon'}\to\mathcal{C}^{|\varepsilon-\varepsilon'|}$
(cf.\,\cite[Lemma B.4]{sato2016equivalences}).
Note that a similar equivalence holds also in the $\mathbb{Z}_{2}$-graded version.

\section{Zhu's algebra and Frenkel--Zhu's bimodule}\label{ZAU}

In this section, we first recall the definition of twisted Zhu's algebra and
Frenkel--Zhu's bimodule originally introduced in \cite{zhu1996modular}
and \cite{frenkel1992vertex}.

\subsection{Definition}\label{DefZhuSec}

Let $V=V^{\bar{0}}\oplus V^{\bar{1}}$
be a vertex operator superalgebra with the $\frac{1}{2}\mathbb{Z}_{\geq0}$-grading
$$V^{\bar{i}}=\bigoplus_{\Delta\in\mathbb{Z}_{\geq0}+\frac{i}{2}}V_{\Delta}\ \ (i\in\{0,1\})$$
with respect to the operator $L_{0}$.
Let $\sigma$ denote the vertex operator superalgebra automorphism of $V$ defined by
$\sigma|_{V^{\bar{i}}}=(-1)^{i}\id_{V^{\bar{i}}}$ for $i\in\{0,1\}$.
Let $\mathcal{M}=\mathcal{M}^{\bar{0}}\oplus\mathcal{M}^{\bar{1}}$ be a 
$\mathbb{Z}_{2}$-graded $g$-twisted $V$-module for $g\in\{\sf id,\sigma\}$.
For $A\in V_{\Delta}\cap V^{\bar{i}}$ and $v\in \mathcal{M}^{\bar{j}}$, we define 
\begin{align*}
&A\underset{g}{*} v:= 
\begin{cases}
\displaystyle
\delta_{i,0}\sum_{\ell=0}^{\infty}\binom{\Delta}{\ell}A_{(\ell-1)}v &\text{ if }g={\sf id},\\
\displaystyle
\sum_{\ell=0}^{\infty}\binom{\Delta}{\ell}A_{(\ell-1)}v &\text{ if }g=\sigma,
\end{cases}\\
&v\underset{g}{*} A:=
\begin{cases}
\displaystyle
\delta_{i,0}\sum_{\ell=0}^{\infty}\binom{\Delta-1}{\ell}A_{(\ell-1)}v &\text{ if }g={\sf id},\\
\displaystyle
(-1)^{ij}\sum_{\ell=0}^{\infty}\binom{\Delta-1}{\ell}A_{(\ell-1)}v &\text{ if }g=\sigma,
\end{cases}\\
&A\underset{g}{\circ} v:=
\begin{cases}
\displaystyle
\sum_{\ell=0}^{\infty}\binom{\Delta-\frac{i}{2}}{\ell}A_{(\ell-2+i)}v &\text{ if }g={\sf id},\\
\displaystyle
\sum_{\ell=0}^{\infty}\binom{\Delta}{\ell}A_{(\ell-2)}v &\text{ if }g=\sigma
\end{cases}
\end{align*}
and extend them by linearity.
We also define $A_{V,g}(\mathcal{M}):=\mathcal{M}/O_{V,g}(\mathcal{M})$,
where $$O_{V,g}(\mathcal{M})
:=\spn_{\mathbb{C}}
\Bigl\{A\underset{g}{\circ}v\,\Big|\,A\in V,\,v\in \mathcal{M}\Bigr\}.$$
Then the next proposition is a natural generalization of \cite[Theorem 2.1.1]{zhu1996modular}
and \cite[Theorem 1.5.1]{frenkel1992vertex} (see e.g.\,\cite{xu1998introduction})

\begin{prp}\label{DefZhu}
The bilinear mappings $\underset{g}{*}\colon V\times\mathcal{M}\to\mathcal{M}$
and $\underset{g}{*}\colon \mathcal{M}\times V\to\mathcal{M}$
give rise to the following structures:
\begin{enumerate}
\item a $\mathbb{Z}_{2}$-graded associateve algebra structure on $A_{V,g}(V)$,
\item a $\mathbb{Z}_{2}$-graded $A_{V,g}(V)$-bimodule structure on $A_{V,g}(\mathcal{M})$.
\end{enumerate}
\end{prp}

Since $I_{c}\underset{g}{\circ}\mathcal{M}=\{0\}$ for any $g$-twisted $L_{c}$-module $\mathcal{M}$,
we simply write $A_{g}(\mathcal{M})$ and $O_{g}(\mathcal{M})$ for $A_{V,g}(\mathcal{M})$
and $O_{V,g}(\mathcal{M})$, respectively.
In this paper we use the notation $[A]$ for the image of $A\in V$
under the natural projection from $V$ to $A_{g}(V)$.

\bibliographystyle{alpha}

\bibliography{ref}

\end{document}